\documentclass[12pt]{amsart}%
\usepackage{amsmath}
\usepackage{amsfonts}
\usepackage{amssymb}
\usepackage{verbatim}
\usepackage{graphicx}%
\providecommand{\U}[1]{\protect\rule{.1in}{.1in}}
\newtheorem{theorem}{Theorem}

\numberwithin{equation}{section}
\begin{document}
\title[Liouville property]{A\ Liouville property for gradient graphs and a Bernstein problem for
Hamiltonian stationary equations. \ }

\begin{abstract}
Using an rotation of Yuan, we observe that the gradient graph of any
semiconvex function is a Liouville manifold, that is, does not admit bounded
harmonic functions. \ As a corollary, we find that any solution of the fourth
order Hamiltonian stationary equation satisfying
\[
\theta\geq\left(  n-2\right)  \frac{\pi}{2}+\delta
\]
for some $\delta>0$ must be a quadratic. \ 

\end{abstract}

\thanks{The author is  partially supported by NSF Grant DMS-1438359. }
\maketitle

\bigskip

\bigskip

In this short note, we record the following.

\begin{theorem}
Suppose that $u$ is semi-convex. \ Then the gradient graph
\[
\Gamma=\left\{  \left(  x,Du(x)\right)  :x\in%
%TCIMACRO{\U{211d} }%
%BeginExpansion
\mathbb{R}
%EndExpansion
^{n}\right\}
\]
with the induced submanifold metric, is a Liouville manifold.
\end{theorem}

\ Recall that a manifold has the Liouville property if all bounded harmonic
functions are constant. \ A rotation of Yuan \cite{YYI} allows us to write the
Laplace operator as a uniformly elliptic divergence operator. \ The result
then follows readily from the De-Giorgi-Nash-Moser theory.

We are interested in studying a fourth order special Lagrangian type equation.
Let $\lambda_{i}$ be the eigenvalues of the Hessian $D^{2}u.$ \ The lagrangian
phase is given by
\[
\theta=\sum_{i=1}^{n}\arctan\lambda_{i}.
\]
We extend the following generalization of theorems of Yuan \cite{YYI}%
\cite{YY06}.

\begin{theorem}
Let $g$ be the metric induced on $\Gamma=\left(  x,Du(x)\right)  .$ \ Suppose
that $u$ is an entire solution of the fourth order Hamiltonian stationary
equation
\begin{equation}
\Delta_{g}\theta=0 \label{Hamstat}%
\end{equation}
with
\begin{equation}
\left\vert \theta(u)\right\vert >\left(  n-2\right)  \frac{\pi}{2}%
+\delta\label{cond}%
\end{equation}
or
\[
D^{2}u\geq0.
\]
Then $u$ is a quadratic function, that is, $\Gamma$ is a plane.
\end{theorem}

For a fixed bounded domain $\Omega\subset$ $%
%TCIMACRO{\U{211d} }%
%BeginExpansion
\mathbb{R}
%EndExpansion
^{n}$ consider the volume functional
\[
F_{\Omega}(u)=\int_{\Omega}\sqrt{\det\left(  I+\left(  D^{2}u\right)
^{T}D^{2}u\right)  }dx.
\]
A function $u$ is critical for $F_{\Omega}(u)$ under compactly supported
variations of the scalar function if and only if $u$ satisfies the equation
(\ref{Hamstat})\ c.f \cite[Proposition 2.2]{SW}. \ In other words, the
gradient graph of $u$ has smallest volume compared with other gradient graphs.
\ \ Recall that if $u$ satisfies the special Lagrangian equation \cite{HL}
\begin{equation}
D\theta=0\label{sLag}%
\end{equation}
then $u$ is critical under \textit{all} variations of the surface.

The Liouville property, together with (\ref{cond}) will force $\theta$ to
satisfy (\ref{sLag}), so $\Gamma$ is a minimal surface. It then follows
immediatly from a result of Yuan that $\Gamma$ is a plane. \ \ For Bernstein
results for (\ref{Hamstat}) with a volume growth constraint, and more
discussion of the problem, see \cite{Mese}.

\section{\bigskip Proof}

\subsection{Proof of Theorem 1}

\ If $u$ is semiconvex, then there exists a value $M$ such that
\[
D^{2}u+MI_{n}\geq0.
\]
It follows that
\[
\arctan\lambda_{i}\geq-\arctan M
\]
for all $\lambda_{i}.$

Letting%
\[
\delta=\frac{\pi}{2}-\arctan M>0,
\]
and
\[
D^{2}u\geq\tan\left(  \delta-\frac{\pi}{2}\right)  .
\]
The Yuan rotation from \cite[section 2]{YY06} is as follows.  \ 

Consider the map
\[
T(x)=\cos\left(  \frac{\delta}{n}\right)  x+\sin\left(  \frac{\delta}%
{n}\right)  Du(x).
\]
Differentiating
\begin{align}
DT &  =\cos\left(  \frac{\delta}{n}\right)  I+\sin\left(  \frac{\delta}%
{n}\right)  D^{2}u(x)\label{11}\\
&  \geq\cos\left(  \frac{\delta}{n}\right)  I+\sin\left(  \frac{\delta}%
{n}\right)  \tan\left(  \delta-\frac{\pi}{2}\right)  I\nonumber\\
&  =\cos\left(  \frac{\delta}{n}\right)  \left(  1+\tan\left(  \frac{\delta
}{n}\right)  \tan\left(  \delta-\frac{\pi}{2}\right)  \right)  I.\nonumber
\end{align}

Recalling the formula
\[
\tan(\alpha-\beta)=\frac{\tan\alpha-\tan\beta}{1+\tan\alpha\tan\beta}%
\]
we see%
\[
1+\tan\left(  \frac{\delta}{n}\right)  \tan\left(  \delta-\frac{\pi}%
{2}\right)  =\frac{\tan\left(  \frac{\delta}{n}\right)  -\tan\left(
\delta-\frac{\pi}{2}\right)  }{\tan\left(  \frac{\delta}{n}-\left(
\delta-\frac{\pi}{2}\right)  \right)  }=\frac{\tan\left(  \frac{\delta}%
{n}\right)  +\tan\left(  \frac{\pi}{2}-\delta\right)  }{\tan\left(  \frac{\pi
}{2}-\delta\frac{n-1}{n}\right)  }.
\]
It follows that%
\[
DT\geq\cos\left(  \frac{\delta}{n}\right)  \frac{\tan\left(  \frac{\delta}%
{n}\right)  +\tan\left(  \frac{\pi}{2}-\delta\right)  }{\tan\left(  \frac{\pi
}{2}-\delta\frac{n-1}{n}\right)  }I>0,
\]
and the map
\[
T:%
%TCIMACRO{\U{211d} }%
%BeginExpansion
\mathbb{R}
%EndExpansion
^{n}\rightarrow%
%TCIMACRO{\U{211d} }%
%BeginExpansion
\mathbb{R}
%EndExpansion
^{n}%
\]
is a diffeomorphism. \ \ 

Next consider the map
\[
\tilde{D}=Du\circ T^{-1}.
\]
By (\ref{11}),
\[
D\tilde{D}(y)=D^{2}u(T^{-1}(y))\left[  \cos\left(  \frac{\delta}{n}\right)
I+\sin\left(  \frac{\delta}{n}\right)  D^{2}u(T^{-1}(y))\right]  ^{-1}.
\]
Diagonalizing $D^{2}u$ at $T^{-1}(y)$ we see
\[
D\tilde{D}|_{y}\leq\max_{i}\frac{\lambda_{i}}{\cos\left(  \frac{\delta}%
{n}\right)  +\sin\left(  \frac{\delta}{n}\right)  \lambda_{i}}\leq\frac
{1}{\sin\left(  \frac{\delta}{n}\right)  }=M_{0}<\infty.
\]
So the map
\[
G:%
%TCIMACRO{\U{211d} }%
%BeginExpansion
\mathbb{R}
%EndExpansion
^{n}\rightarrow\Gamma\subset%
%TCIMACRO{\U{211d} }%
%BeginExpansion
\mathbb{R}
%EndExpansion
^{n}\times%
%TCIMACRO{\U{211d} }%
%BeginExpansion
\mathbb{R}
%EndExpansion
^{n}%
\]
given by
\[
G(x)=\left(  T^{-1}(x),Du\circ T^{-1}\right)
\]
is a diffeomorphism onto the gradient graph $\Gamma.$ Thus the pulled-back
metric is given by
\[
g=I_{n}+G^{\ast}\bar{g}%
\]
and satisfies
\[
I_{n}\leq g\leq(1+M_{0}^{2})I_{n}.
\]
It follows that the Laplacian, given by
\[
\Delta_{g}f=\frac{\partial_{j}\left(  \sqrt{\det g}g^{ij}\partial_{i}f\right)
}{\sqrt{\det g}},
\]
is a uniformly elliptic divergence type operator. \ 

The remaing proof is standard, but we include it for completeness. \ 

We recall the Harnack inequality of De-Giorgi-Nash-Moser (cf \cite[Theorem
8.20]{GT})

\begin{theorem}
Let $u\geq$ $0$ be a solution of
\[
\partial_{j}\left(  a^{ij}(x)\partial_{i}f(x)\right)  =0
\]
on $B_{4}(0)$ with
\[
0<\varepsilon I_{n}\leq a^{ij}\leq\frac{1}{\varepsilon}I_{n}.
\]
There exists a constant $C$ such that
\[
\sup_{B_{1}(0)}f\leq C\inf_{B_{1}(0)}f.
\]

\end{theorem}

We may assume that either $f$ or ($-f$ )is bounded below, and we may add a
constant and assume $\inf f=0.$ \ Notice that
\[
f_{R}(x)=f\left(  \frac{x}{R}\right)
\]
is a solution of the equation
\[
\partial_{j}\left(  a^{ij}\left(  \frac{x}{R}\right)  \partial_{i}%
f_{R}(x)\right)  =0
\]
so satisfies the hypothesis of the Harnack inequality. In particular%
\[
\sup_{B_{R}(0)}f\leq C\inf_{B_{R}(0)}f
\]
for every ball $B_{R}(0)$ with a fixed constant $C.$ \ \ Taking $R\rightarrow
\infty$ gives $\sup f=0.$

In fact, we can state a slightly more general theorem.

\begin{theorem}
Suppose that $F(D^{2}u)$ is an elliptic functional, and let $g$ be the induced
metric on the gradient graph. \ If $u$ is semi-convex and
\[
\Delta_{g}F(D^{2}u)=0
\]
then
\[
F(D^{2}u)=\text{const. }%
\]

\end{theorem}

\begin{proof}
If $u$ is semiconvex, then the exists a value $M$ such that
\[
D^{2}u-MI_{n}\geq0.
\]
It follows by ellipticity that
\[
F(D^{2}u)\geq F(MI_{n})>-\infty.
\]
The result follows immediately from our main theorem. \ 
\end{proof}

\bigskip

\subsection{Proof of Theorem 2}

The function $\theta$ is odd in $u$, so we need only show that
\begin{equation}
\theta(u)>\left(  n-2\right)  \frac{\pi}{2}+\delta
\end{equation}
implies that $u$ is semiconvex. \ If
\[
\lambda_{i}<\arctan(\delta-\frac{\pi}{2})
\]
then we must have
\[
\sum_{i\neq j}\arctan\lambda_{j}>\left(  n-2\right)  \frac{\pi}{2}%
+\delta-(\delta-\frac{\pi}{2})=\left(  n-1\right)  \frac{\pi}{2}%
\]
which is clearly a contradicton as%
\[
\arctan\lambda_{j}\leq\frac{\pi}{2}.
\]
We conclude that
\[
D^{2}u\geq\arctan(\delta-\frac{\pi}{2})
\]
and $u$ is semiconvex. \ Thus $\theta(u)=$ const, by Theorem 1. \ \ The result
follows by the main results in \cite{YYI}\cite{YY06}.

\bigskip

\bibliographystyle{amsalpha}
\bibliography{hamstat}

\end{document}